\documentclass[a4paper,12pt,reqno]{amsart}
\usepackage[utf8]{inputenc}
\usepackage[T1]{fontenc}
\usepackage{lmodern}
\usepackage[english]{babel}
\usepackage{microtype}
\usepackage[normalem]{ulem}
\usepackage{comment}
 
\usepackage{amsmath,amssymb,amsfonts,amsthm,esint}
\usepackage{mathtools,accents}
\usepackage{mathrsfs}
\usepackage{aliascnt}
\usepackage{braket}
\usepackage{bm}

\usepackage[a4paper,margin=3cm]{geometry}
\usepackage[citecolor=blue,colorlinks]{hyperref}

\usepackage{enumerate}
\usepackage{xcolor}
\usepackage{graphicx}
\usepackage{mathbbol}

\usepackage{hyperref}
\usepackage{graphicx, wrapfig}

\usepackage{txfonts}
\usepackage{ulem, cancel}
\usepackage{color}

\usepackage{enumitem}
\usepackage{xcolor}

\counterwithin{figure}{section}
\theoremstyle{plain} 

\newtheorem{thm}{Theorem}[section]
\newcommand{\bt}{\begin{thm}}
\newcommand{\et}{\end{thm}}

\newtheorem{cor}[thm]{Corollary}

\newcommand{\bc}{\begin{cor}}
\newcommand{\ec}{\end{cor}}

\newtheorem{lem}[thm]{Lemma}   

\newcommand{\bl}{\begin{lem}}
\newcommand{\el}{\end{lem}}

\newtheorem{prop}[thm]{Proposition}
\newcommand{\bp}{\begin{prop}}
\newcommand{\ep}{\end{prop}}

\newtheorem{defn}[thm]{Definition}

\newcommand{\ben}{\begin{itemize}}
\newcommand{\een}{\end{itemize}}

\newcommand{\bd}{\begin{defn}}        
\newcommand{\ed}{\end{defn}}

\newtheorem{rmrk}[thm]{Remark}   

\newcommand{\br}{\begin{rmrk}}
\newcommand{\er}{\end{rmrk}}

\newtheorem{example}[thm]{Example}

\newcommand{\tHto}{\stackrel { \tau-\textrm{H}}{\longrightarrow} }

\newcommand{\tGHto}{\stackrel { \tau-\textrm{GH}}{\longrightarrow} }

\newcommand{\be}{\begin{equation}}

\newcommand{\ee}{\end{equation}}

\newcommand{\R}{\mathbb{R}}

\newcommand{\dist}{\operatorname{dist}}
\newcommand{\diam}{\operatorname{diam}}

\newcommand{\disjointunion}{\sqcup}

\DeclareMathOperator{\interior}{int}

\def\iff{\Longleftrightarrow}

\title[Convergence of Timed-Metric Spaces and Causality]{Convergence of Timed-Metric Spaces and Causality}

\author[Che]{Mauricio Che} 
\author[Perales]{Raquel Perales}

\thanks{R. Perales was funded by 
the Austrian Science Fund
(FWF) [Grant DOI: 10.55776/EFP6]. M.~Che was funded by the Austrian Science Fund (FWF) [Grant DOI: 10.55776/STA32]. 
}

\begin{document}

\begin{abstract}  
We introduce the notion of timed Gromov–Hausdorff distance for timed-metric spaces. We prove that this distance is bilipschitz equivalent to the intrinsic timed-Hausdorff distance of Sakovich–Sormani, and therefore induces the same notion of convergence. We establish a compactness theorem for the timed Gromov–Hausdorff distance, obtained as a straightforward consequence of Gromov’s classical compactness theorem.

We then investigate the causal structure of timed-metric spaces and the stability of causality under intrinsic timed-Hausdorff convergence. We further analyze causally-null timed-metric spaces and develop several of their basic properties. As a curiosity, we include in an appendix Gromov’s original proof of his compactness theorem, as presented in his paper on groups of polynomial growth.
\end{abstract}

\maketitle
{
\hypersetup{linkcolor=blue}
\setcounter{tocdepth}{1}
\tableofcontents
}

\section{Introduction}\label{sec:intro}

We recall that a space-time consists of a time-oriented Lorentzian manifold, and that the development of notions of low-regularity space-times, as well as of weak convergence in this setting, has been an active area of research in recent years; see, for example, \cite{Kunzinger-Saemann,Minguzzi-Suhr-24,Cavalletti-Mondino,Cavalletti-Manini-Mondino-Optimal}. In particular, Sakovich--Sormani introduced the notion of a timed-metric space \cite{SakSor-Notions} and defined several distances between such spaces, including the intrinsic timed-Hausdorff distance.

In this manuscript, we work within the framework of Sakovich--Sormani, and we now recall some results in this setting. Given a space-time, Sakovich--Sormani \cite{SakSor-Notions} used the cosmological time of Andersson--Galloway--Howard \cite{AGH} together with the null distance of Sormani-Vega \cite{SV-Null} to associate a timed-metric space to any smooth causally null-compactifiable space-time. They further proved that the intrinsic timed-Hausdorff distance is definite, in the sense that the distance between two causally null-compactifiable space-times vanishes if and only if there exists a bijection preserving both the distance and the time. In the smooth setting, this is equivalent to the existence of a Lorentzian isometry between the original manifolds, by results of Sakovich--Sormani \cite{SakSor-Null}, Burtscher--Garcia Heveling \cite{Burtscher-Garcia-Heveling-Global}, and Galloway \cite{Galloway-causal}. Moreover, Che--Perales--Sormani \cite{Che-Perales-Sormani-2025} established a compactness theorem for the intrinsic timed-Hausdorff distance using addresses, which we will recall later; Perales showed in \cite{perales-intrinsic-timed-hausdorff} that, as conjectured by Sakovich--Sormani, intrinsic timed-Hausdorff convergence implies convergence with respect to some of the other notions introduced by Sakovich--Sormani; and in the upcoming work \cite{prados-zoghlami}, conditions on space-times implying intrinsic timed-Hausdorff convergence will be addressed.  We now proceed to state our main results.

\begin{defn}\label{defn:timed-metric-space}
A timed-metric space, 
$(X,d,\tau)$, is a metric space,
$(X,d)$,
endowed with a $1$-Lipschitz function,  
$
\tau:X\to [0,\infty),
$
which we call time function. 
\end{defn}

A useful example of a timed-metric space is the max-product space, implicitly defined in \cite[Propositions~4.24 and 4.25]{SakSor-Notions}.

\begin{defn}\label{def:product space}
Given a metric space $(Z,d_Z)$ and $\tau_{max}>0$, we define 
the max-product space of the compact interval, $[0,\tau_{max}]$, with
$(Z,d_Z)$
as follows:
\be \label{eq:max-prod0}
\left([0,\tau_{max}]\times Z, \,d_{prod}, \,\tau_{prod}\right),
\ee
where
\be \label{eq:max-prod1}
\tau_{prod}(t,z)=t 
\ee
and
\be
\label{eq:max-prod-2}
 d_{prod}\Big((t_1,z_1),(t_2,z_2)\Big)=\max\Big\{|t_1-t_2|,\, d_Z(z_1,z_2)\Big\}.
\ee
\end{defn}

We now use the max-product space construction to define an analogue of the Gromov--Hausdorff distance for timed-metric spaces.

\begin{defn}\label{defn:t-H-new}
The {\bf timed-Gromov--Hausdorff distance} between two 
compact timed-metric spaces is
\begin{equation*}
d_{\tau-GH}\bigg((X_1,d_1,\tau_1),(X_2,d_2,\tau_2)\bigg)=
\inf d_H^{[0,\tau_{max}]\times Z}\bigg(\varphi_1(X_1),\varphi_2(X_2)\bigg)
\end{equation*}
where the infimum is over all compact metric spaces, $Z$, and over all time and distance preserving maps,
\[
\varphi_j: (X_j,d_j,\tau_j)\to
([0,\tau_{max}]\times Z, d_{prod}, \tau_{prod}),
\]
where $
\tau_{max}=\max\{\tau_{1,\max},\tau_{2,\max}\}$ and $\tau_{j,\max}=\max\{\tau_j(x):\, x\in X_j\}$.
\end{defn}

The
timed-Gromov--Hausdorff distance between two 
compact timed-metric spaces is symmetric and
takes finite values.  In fact,
\[
d_{\tau-GH}\bigg((X_1,d_1,\tau_1),(X_2,d_2,\tau_2)\bigg)\le
d_{GH}\bigg((X_1,d_1),(X_2,d_2)\bigg)+\tau_{max}.
\]
In Theorem~\ref{thm:timed-H-equiv}, we prove that this notion is bi-Lipschitz equivalent to 
the intrinsic timed-Hausdorff distance defined by
Sakovich--Sormani in \cite[Definition~4.22]{SakSor-Notions}. This equivalence was first conjectured in \cite[Remark 4.26]{SakSor-Notions}.
Since Sakovich--Sormani's distance is definite,  our timed-Gromov--Hausdorff distance is definite as well, this means that
\begin{equation*}
d_{\tau-GH}\bigg((X_1,d_1,\tau_1),(X_2,d_2,\tau_2)\bigg)=0
\end{equation*}
if and only if there is a distance and time preserving bijection
between the timed-metric spaces.  We prove in Theorem~\ref{thm:triangle} that
this notion also satisfies the triangle inequality. In the Appendix \ref{sec-appendix}, we reprove Gromov's Compactness Criterion from \cite{Gromov-poly}, and as a straightforward application, we prove a compactness theorem for timed Gromov--Hausdorff distance. However, Che--Perales--Sormani's compactness result for intrinsic timed-Hausdorff distance \cite{Che-Perales-Sormani-2025} can also be seen as a compactness theorem for timed Gromov--Hausdorff distance.

\begin{thm}[Compactness Theorem]\label{thm:timed-GH-compactness}
If $(X_j,d_j,\tau_j)$ is a sequence of compact timed-metric spaces that are uniformly equibounded
and equicompact, and there exists $\tau_{\max} \geq 0$ such that 
\be
\bigcup_{j \in \mathbb N} \tau_j(X_j)\subset [0,\tau_{max}].
\ee
Then a subsequence converges in the timed-Gromov--Hausdorff sense to a compact timed-metric space $(X_\infty, d_\infty, \tau_\infty)$. In fact, there exist
compact $Z$ and distance and time preserving maps,
\be\label{eq-TDmaps}
\varphi_j= (\tau_j, \xi_j): X_j \to [0,\tau_{max}]\times Z 
\ee
where $\xi_j:X_j \to Z$ is {\bf distance preserving}, such that the Hausdorff distance between the images converges to zero:
\be\label{eq-HausReal}
d_H^{[0,\tau_{max}]\times Z}\bigg(\varphi_j(X_j),\varphi_\infty(X_\infty)\bigg)\to 0.
\ee
\end{thm}

Similarly to the convergence of points under the Gromov–Hausdorff distance, we obtain convergence of both points and ``times'' under the timed-Gromov--Hausdorff and 
intrinsic timed-Hausdorff distances, c.f. Remark \ref{rmrk-pointandtimeC}.

\begin{cor}\label{cor:pointtimeConvergence}
Suppose a sequence of compact 
timed-metric spaces converges 
in the intrinsic timed-Hausdorff sense to a compact timed-metric space,
\begin{equation*}
(X_j,d_j,\tau_j)\tGHto
(X_\infty,d_\infty,\tau_\infty).
\end{equation*}
Then for a subsequence $j_k$ the conclusions of Theorem~\ref{thm:timed-GH-compactness}
hold and the following is satisfied. 
Given any sequence of points $p_j \in X_j$, there exists $p_\infty \in X_\infty$ and a subsequence, that we do not relabel, such that 
\be\label{eq-pointConv}
\varphi_j(p_j)=(\tau_j(p_j), \xi_j(p_j))\to \varphi_\infty(p_\infty)= (\tau_\infty(p_\infty), \xi_\infty(p_\infty)).
\ee
Conversely, given $p_\infty \in X_\infty$ there is a sequence $p_j\in X_j$ such that 
\be\label{eq-pointConv2}
\varphi_j(p_j) \to \varphi_\infty(p_\infty).
\ee
\end{cor}

\bigskip

Since causality plays a crucial role in the Lorentzian setting, we define a causal structure in our setting.

\begin{defn}\label{defn:causal-structure}
For a timed-metric space $(X,d,\tau)$ we define a \textbf{causal structure} as follows: For $p,q \in X$ we say that $p$ is in the \textbf{causal future} of $q$, written  $p\in J^+(q)$, iff
\begin{equation}\label{def-causal-struc}
\tau(p)-\tau(q)=d(p,q).
\end{equation}
If this holds, we also say that $q$ is in the \textbf{causal past} of $p$,
written $q\in J^{-}(p)$. When necessary, we write $J^{+}_X$ and $J^{-}_X$ in each case. We say that $p$ and $q$ are \textbf{causally related} if 
$p\in J^+(q)$ or $q\in J^+(p)$.
\end{defn}

As a consequence of Theorem \ref{thm:timed-GH-compactness} we obtain the following corollary describing how causality behaves under timed-Gromov--Hausdorff convergence. The first part shows stability of causality.

\begin{cor}\label{cor:CausalSeq}
Suppose a sequence of compact 
timed-metric spaces converges 
in the timed-Gromov--Hausdorff sense to a compact timed-metric space,
\be
(X_j,d_j,\tau_j)\tGHto
(X_\infty,d_\infty,\tau_\infty).
\ee
Then for a subsequence $j_k$
the conclusions of Theorem~\ref{thm:timed-GH-compactness}
hold and the following is satisfied:

\begin{enumerate}
\item\label{item:CausalSeq1} Given two sequences $p_j,q_j\in X_j$ such that 
$p_j\in J^+_{X_j}(q_j)$, there exists 
$p_\infty,q_\infty \in X_\infty$ such that
\be\label{eq-CausConv}
p_\infty\in J^+_{X_\infty}(q_\infty)
\ee
and $\varphi_j(p_j) \to \varphi_\infty(p_\infty)$, $\varphi_j(q_j) \to \varphi_\infty(q_\infty)$.
\item\label{item:CausalSeq2} If $p_\infty\notin J^+_{X_\infty}(p'_\infty)$
then there exist two sequences of points 
$p_j, p'_j \in X_j$ such that
$\varphi_j(p_j) \to \varphi_\infty(p_\infty)$ and 
$\varphi_j(p'_j) \to \varphi_\infty(p'_\infty)$, and for 
 $j$ sufficiently large it holds, 
\be
p_j\notin J^+_{X_j}(p'_j).
\ee
\end{enumerate}
\end{cor}

\begin{rmrk}\label{rmrk-pointandtimeC}
In Corollary \ref{cor:pointtimeConvergence}
and Corollary \ref{cor:CausalSeq} we do have convergence of points in the following sense. Since the maps $\xi_j:X_j \to Z$ are distance preserving, it is clear that 
for any sequence $x_k \in X_j$, $k \in \mathbb N$, and $x \in X_j$,
\[
x_k \to x \iff \xi_j(x_k) \to \xi_j(x).
\]
Note that by  
Lemma \ref{lem:zeta-dist-pres}, the class of functions considered in the definition of $d_{\tau\text{-}GH}$ is larger than the class consisting of all functions $(\tau,\zeta): X \to [0,\tau_{\max}] \times Z$ for which $\zeta$ is distance-preserving. The proofs of both corollaries also hold under 
intrinsic timed-Hausdorff distance, since in this case the class of functions 
used to calculate $d_{\tau\text{-}H}$ consists precisely of all maps $(\tau,\kappa): X \to [0,\tau_{\max}] \times Z$ with $\kappa$ a Fr\'echet map, which is distance-preserving.  Moreover, when using addresses in the sense of Definition~\ref{defn:addresses}, the sequences appearing in the final parts of these corollaries
are automatically given by the properties of the addresses,  c.f. Corollary \ref{cor:CausalSeqWithAddressesItem2}.

\end{rmrk}

Using the causal structure defined above, we can also define an associated \textbf{null distance}, in analogy to 
\cite{SV-Null}. We prove that this null distance is idempotent, whenever it defines a metric. 

\begin{prop}\label{prop-null-distance is greater}
Let $(X,d,\tau)$ be a timed-metric space. 
Then the null distance $\hat{d}_{d,\tau}$, given in Definition~\ref{defn:null-distance}, is an extended distance such that
\begin{equation}\label{eq:null-distance is greater}
\hat{d}_{d,\tau}\geq d,
\end{equation}
and $\tau$ is $1$-Lipschitz with respect to $\hat{d}_{d,\tau}$.
Moreover, if $\hat{d}_{d,\tau}$ only takes finite values then $(X,\hat{d}_{d,\tau},\tau)$ is a timed-metric space and $\hat{d}_{\hat{d}_{d,\tau},\tau} = \hat{d}_{d,\tau}$.
\end{prop}

When a timed-metric space $(X,d,\tau)$ is such that $\hat{d}_{d,\tau}=d$, we say that $(X,d,\tau)$ is \textbf{causally null}. Thus, Proposition~\ref{prop-null-distance is greater}
says that, for any timed-metric space $(X, d, \tau)$ such that $\hat{d}_{d,\tau}$ is a distance, the timed-metric space $(X, \hat{d}_{d,\tau},\tau)$ is causally
null.

\bigskip 
Our paper is organized as follows.
In Section~\ref{sec:Background-on-metric}, we review Gromov--Hausdorff distance, Fr\'echet maps, timed-metric spaces, timed-Fr\'echet maps, intrinsic timed-Hausdorff distance and Gromov's precompactness for intrinsic timed-Hausdorff distance. In Section~\ref{sec-timedGH}, we prove Theorem~\ref{thm:timed-H-equiv}, which asserts that 
Definition~\ref{defn:t-H-new} is equivalent to the notion previously introduced by Sakovich--Sormani in \cite[Definition~4.22]{SakSor-Notions}. We also show that the triangle inequality holds for this new notion. Further, we prove Theorem \ref{thm:timed-GH-compactness}.

In Section~\ref{s:understanding causality} we show that the causal structure induced by a timed-metric space defines a partial order, 
give the proof of Corollary~\ref{cor:CausalSeq}, and reformulate  both second items of Corollary \ref{cor:pointtimeConvergence} and Corollary \ref{cor:CausalSeq}, for the intrinsic timed-Hausdorff distance using addresses in the sense of Definition~\ref{defn:addresses}. Moreover,
in Propositions~\ref{prop:causal curves are realizers} and \ref{prop:piecewise causal realizers},
we establish two results that are analogous to \cite[Corollary~3.19]{SV-Null} and to part of \cite[Lemma~3.20]{SV-Null}, concerning curves that are distance realizers in causally-null timed-metric spaces.

In Section \ref{s:causally-null tms} we study the null distance and causally-null timed-metric spaces. In particular, we give the proof of Proposition \ref{prop-null-distance is greater}. 
Then we see that
 \cite[Lemma~3.5]{Kunzinger-Steinbauer-22} 
can be adapted to provide a sufficient condition for obtaining a causally-null timed-metric space. 
We conclude the paper proving 
Gromov's Compactness and Embedding Theorem \cite{Gromov-poly}, which we use to prove Theorem \ref{thm:timed-GH-compactness}.

\bigskip

{\bf Acknowledgements:}  The authors would like to thank Christina Sormani for very fruitful discussions and encouraging them to explore this project. The second named author gratefully acknowledges support from the Simons Center for Geometry and Physics where she had the opportunity to discuss in person with Sormani as part of the Program ``Geometry and Convergence in Mathematical General Relativity'' in September 2025. 
The authors were funded by the Austrian Science Fund
(FWF) [Grant DOI's: 10.55776/EFP6, 10.55776/STA32]. For open access purposes, the authors have applied a CC BY public copyright license to any author-accepted manuscript version arising from this submission.

\section{Background}\label{sec:Background-on-metric}

We briefly review 
basic material about metric spaces, see \cite{BBI} for details. Moreover, we state 
results and definitions from Sakovich--Sormani in \cite{SakSor-Notions} and Che--Perales--Sormani in
\cite{Che-Perales-Sormani-2025} that we will use.

\subsection{Metric spaces and GH distance}

Given a metric space $(Z,d_Z)$, the Hausdorff distance between two subsets $U,W\subset Z$
is given by
\[
d_H^Z\bigg(U,W\bigg)=\inf
\left\{ r>0:  
\begin{array}{c} \forall u\in U, \exists w\in W 
\textrm{ s.t. }
d_Z(u,w)<r\\
\forall w\in W, \exists u\in U
\textrm{ s.t. }
d_Z(u,w)<r
\end{array}
\right\}.
\]

Given two compact metric spaces $(X_j,d_j)$, $j=1,2$, 
the Gromov--Hausdorff (GH) distance between them is
defined by
\[
d_{GH}\bigg((X_1,d_1),(X_2,d_2)\bigg)=\inf d_H^Z\bigg(\varphi_1(X_1),\varphi_2(X_2)\bigg)
\]
where the infimum is over all
 distance preserving maps,
$\varphi_j:X_j\to Z$, into all possible metric spaces, $Z$. 
Furthermore, a map between metric-spaces, 
$\varphi:(X,d_X)\to (Z,d_Z)$, is said to be
distance preserving if
\[
d_Z(\varphi(x),\varphi(x'))=d_X(x,x') \qquad \forall x,x'\in X.
\]
A particular class of distance preserving maps
are Fr\'echet maps, to define them recall that $(\ell^\infty, d_{\ell^\infty})$
is the Banach space defined as 
\be\label{eq:defn-ell-infty}
\ell^\infty=\{(s_1,s_2,...)\,: s_i \in {\mathbb{R}}, \, d_{\ell^\infty}((s_1,s_2,...).(0,0,...))<\infty\}
\ee
where   
\be\label{eq:d-ell-infty}
d_{\ell^\infty}((s_1,s_2,...),(r_1,r_2,...))
=\sup\{|s_i-r_i|\,:\, i\in {\mathbb N}\}.
\ee

\begin{defn}\label{defn:Frechet}
Given a separable metric-space, $(X,d)$, 
for any countable dense collection of points $\mathcal{N}=\{x_1,x_2,\ldots\}$ of $X$, 
the map
\[
\kappa_X=\kappa_{X,\mathcal{N}}: (X,d)\to (\ell^\infty, d_{\ell^\infty})
\]
given by
\be\label{eq:Frechet}
\kappa_X(x)=(d(x_1,x),d(x_2,x),...)\subset \ell^\infty.
\ee
is called a Fr\'echet map. 
\end{defn}

Fr\'echet maps are distance-preserving (see \cite{Frechet} and \cite{SakSor-Notions}). 
Adopting the notation of Sakovich--Sormani, given $(X,d_X)$ and $(Y,d_Y)$ compact metric spaces, let
\begin{equation*}
d_{\kappa\text{-}GH}\big((X,d_X),(Y,d_Y)\big)
:= \inf \, d_H^{\ell^\infty}\big(\kappa_X(X), \kappa_Y(Y)\big),
\end{equation*}
where the infimum is taken over all pairs of Fréchet maps $\kappa_X$ and $\kappa_Y$. 
By \cite[Proposition~2.7]{SakSor-Notions}, 
$d_{GH}$ is biLipschitz equivalent to $d_{\kappa  \text{-}GH}$.
In analogy with this, in Theorem \ref{thm:timed-H-equiv} we show bilipschitz 
equivalence between $d_{\tau-H}$ and $d_{\tau-GH}$.

\subsection{Timed-metric spaces and intrinsic timed-Hausdorff distance}

We now briefly review some results and definition from Sakovich--Sormani in
\cite{SakSor-Notions}:

\begin{defn}\label{defn:tau-K}
Given a compact timed-metric space, $(X,d,\tau)$
and a countably dense collection of points,
${\mathcal N}$,
we define the
timed-Fr\'echet map as 
\[
\kappa_{\tau,X}
=\kappa_{\tau,X,{\mathcal N}}: X \to [0,\tau_{max}]\times \ell^\infty\subset \ell^\infty,
\qquad 
\kappa_{\tau,X}=
(\tau,\kappa_{X,{\mathcal N}}).
\]
\end{defn}

\begin{prop}[\cite{SakSor-Notions}]
\label{prop:tau-K-dist-pres}
Any
timed-Fr\'echet map, 
$
\kappa_{\tau,X}: X \to [0,\tau_{max}]\times \ell^\infty\subset \ell^\infty,
$
is distance preserving.
\end{prop}

\begin{defn}[{\cite[Definition~4.22]{SakSor-Notions}}]\label{defn:timed-H}
The {\bf
intrinsic 
timed-Hausdorff
distance}
between two compact timed-metric spaces,
$(X_i, d_i, \tau_i)$, $i=1,2$, is given by
\be \label{eq:tK-GH-1}
d_{\tau-H}
\Big((X_1,d_1,\tau_1),(X_2,d_2,\tau_2)\Big)=
\inf\, d^{\ell^\infty}_H(
\kappa_{\tau_1,X_1}(X_1),
\kappa_{\tau_2,X_2}(X_2))
\ee
where the infimum is taken over all possible 
timed-Fr\'echet maps,
$\kappa_{\tau_1,X_1}:X_1\to \ell^\infty$ and
$\kappa_{\tau_2,X_2}:X_2\to \ell^\infty$, which are
found by considering all countable dense collections of points in the $X_i$ and all reorderings of these collections of points.
\end{defn}

The intrinsic timed-Hausdorff 
distance satisfies the triangle inequality (see \cite[Proposition~4.4]{Che-Perales-Sormani-2025}) and is definite in the following sense. 

\begin{thm}[{\cite[Theorem 5.14]{SakSor-Notions}}]
\label{thm:definite-timed-H}
Let
$(X_i, d_i, \tau_i)$, $i=1,2$, be 
two compact timed-metric spaces.
Then, 
\be \label{eq:tau-K-definite-1}
d_{\tau-H}\Big((X_1,d_1,\tau_1),(X_2,d_2,\tau_2)\Big)=0
\ee
if and only if there is a 
distance and time preserving bijection
\be\label{eq:tau-K-definite-2}
\varphi: (X_1,d_1,\tau_1) \to (X_2,d_2,\tau_2).
\ee
\end{thm}

We note that a map between two timed-metric spaces, $\varphi: (X_1,d_1,\tau_1) \to (X_2,d_2,\tau_2)$, is said to be
\textit{time preserving} if $\tau_2(\varphi(x))=\tau_1(x)$ for all $x\in X_1$.

\subsection{Correspondences}

Let us recall the notion of \textit{correspondences} between sets, which yields a standard characterization of bounds for the Hausdorff distance. The results here will be used in the proof of Theorem~\ref{thm:timed-H-equiv}.

\begin{defn}\label{defn:Corr}
Given two sets $A$ and $B$, 
a correspondence between them, denoted by $\mathcal{C}$, is a set $\mathcal{C} \subset A \times B$ such that 
\[
\forall a\in A,\,
\exists b\in B \, s.t. \,
 (a,b)\in {\mathcal C}
\]
and
\[
\forall b\in B,\,
\exists a\in A \, s.t. \,
(a,b)\in {\mathcal C}.
\]
Additionally, if $(X,d_X)$ and $(Y,d_Y)$ are metric spaces, such that 
$A \subset X$ and $B \subset Y$, the distortion of 
$\mathcal{C}$ is defined as 
\[
\dist(\mathcal C)= \sup \{|d_X(a,a') - d_Y(b,b') \, | \, (a,b), (a',b') \in \mathcal C \}.
\]
\end{defn}

\begin{defn}\label{defn:composition of correspondences}
Let $\mathcal{C}\subset A\times B$ and $\mathcal{C}'\subset B\times C$ be correspondences. Define $\mathcal{C}''\subset A\times C$ by
\[
\mathcal{C}'' = \{(a,c)\in A\times C: \exists b\in B\ \text{such that}\ (a,b)\in\mathcal{C}\ \text{and}\ (b,c)\in \mathcal{C}'\}.
\]
Then it is immediate to verify that $\mathcal{C}''$ is a correspondence between $A$ and $C$. This is the \textit{composition} of correspondences $\mathcal{C}$ and $\mathcal{C}'$, and we denote it by $\mathcal{C}'\circ \mathcal{C}$. 
\end{defn}

By the triangle inequality, it follows that 
\be\label{eq:dis of composition}
\dist(\mathcal{C}'\circ \mathcal{C}) \leq \dist(\mathcal{C}') + \dist(\mathcal{C}).
\ee

\begin{prop}\label{prop:Hausdorff}
Let $Z$ be a metric space. 
Given two subsets $A,B\subset Z$, it holds
\[
d_H^Z(A,B)<R
\]
if and only if
\[
\mathcal{C}=\Bigl\{(a,b) \,|\, d_Z(a,b)<R \Bigr\}
\subset A\times B
\]
is a correspondence between $A$ and $B$. Moreover, in this case $\dist(\mathcal C) < 2R$.
\end{prop}

The following lemma will be used in the proof of Theorem~\ref{thm:timed-H-equiv} to establish the equivalence of different definitions of timed-Hausdorff distance between timed-metric spaces.

\begin{lem} \label{lem:corr-K}
Let $(X, d_X)$ and $(Y, d_Y)$ be two compact metric spaces.
If there is a correspondence $\mathcal C$ such that $\dist(\mathcal C) < R$, then there exist a pair of
Fr\'echet maps
$\kappa_X:X \to \ell^\infty$
and
$\kappa_Y:Y \to \ell^\infty$
such that for all $(x,y) \in \mathcal C$
\be \label{eq:corr-K}
d_{\ell^\infty}(\kappa_X(x),\kappa_Y(y))<R, 
\ee
and thus, $d_H^{\ell^\infty}(\kappa_X(X),\kappa_Y(Y))<R$.
\end{lem}

\begin{proof}
Since $X$ and $Y$ are compact,
there exists $R'<R$ such that 
$\dist(\mathcal C)< R'$.
Let $\{x'_1,x'_2,...\} \subset X$
and $\{y'_1,y'_2,...\} \subset Y$
be countably dense collection of points.
For each $y_i'$ take
\be\label{eq:chxi}
x_{y_i'}\in X \, s.t.\, (x_{y_i'},y_i')\in \mathcal{C}\subset X\times Y 
\ee
and, analogously, for each $x_i'$ take
\be\label{eq:chyi}
y_{x_i'}\in Y \, s.t. \,  (x'_i,y_{x_i'})\in \mathcal{C}\subset X\times Y.
\ee
Define
\[
{\mathcal N}_X=\{x_1,x_2,x_3...\}=\{x_1',x_{y_1'},x_2',x_{y_2'},
x_3',x_{y_3'},...\}\subset X
\]
and 
\[
{\mathcal N}_Y=\{y_1,y_2,y_3,....\}=
\{y_{x_1'},y_1',y_{x_2'}, y_2',
y_{x_3'},y_3',...\}\subset Y.
\]
Then the Fr\'echet maps $\kappa_{X}= \kappa_{X,{\mathcal N}_X}:X\to \ell^\infty$ 
and 
$\kappa_{Y}=\kappa_{Y,{\mathcal N}_Y}:Y\to \ell^\infty$  are given by 
\begin{eqnarray*}
\kappa_X(x)&=&(d_X(x_1',x),d_X(x_{y_1'},x),d_X(x_2',x),
d_X(x_{y_2'},x),\dots),\\
\kappa_Y(y)&=&
(d_Y(y_{x_1'},y),d_Y(y_1',y),
d_Y(y_{x_2'},y), d_Y(y_2',y),\dots).
\end{eqnarray*}
If $(x,y)\in\mathcal{C}$, 
then  by \eqref{eq:chxi} the odd coordinates of 
$\kappa_X(x)-\kappa_Y(y)\in \ell^\infty$ satisfy 
\[
|d_X(x_i',x)-d_Y(y_{x_i'},y)| 
\leq R'.
\]
Analogously, by 
\eqref{eq:chyi} the even coordinates satisfy
\[
|d_X(x_{y_i'},x)-d_Y(y_i',y)|\leq R'.
\]
Hence, each coordinate of 
$\kappa_X(x)-\kappa_Y(y)\in \ell^\infty$, 
is smaller than or equal to $R'$. 
Therefore, 
\[
d_{\ell^\infty}(\kappa_X(x),\kappa_Y(y))\le R'
<R.
\]
From this, and recalling that $\mathcal C$ is a correspondence, we get 
$d_H^{\ell^\infty}(\kappa_X(X),\kappa_Y(Y))<R$.
\end{proof}

\subsection{Che--Perales--Sormani's compactness theorem}\label{sect:p1}

Recall that a sequence $\{(X_j,d_j)\}_{j\in\mathbb{N}}$ of metric spaces is \textit{equibounded} if the diameters are uniformly bounded
\be\label{eq:equibounded-B}
\exists D>0 \,s.t.\, \forall j\in\mathbb{N} : \diam_{d_j}(X_j)\le D,
\ee
and \textit{equicompact} if the spaces are uniformly compact
\be\label{eq:equicompact-B}
\forall R\in (0,D]\,
\,\exists N(R) \in \mathbb N \,s.t.\, \forall j\in\mathbb{N}\,\,
\exists x^j_{R,i}\in X_j\,\, s.t.\,\,
X_j\subset \bigcup_{i=1}^{N(R)} B_{d_j}(x^j_{R,i},R).
\ee

\begin{prop}[\cite{Che-Perales-Sormani-2025}]\label{prop:selection}
Let $\{(X_j,d_j)\}_{j\in\mathbb{N}}$ be an equibounded and equicompact sequence of metric spaces. Define $\varepsilon_i = 2^{-i}$ and $N_i=N(\varepsilon_i)$. 
\bigskip
For each $i\in {\mathbb N}$, let $A_i$ be the finite set, 
\be \label{eq:index}
A_i= \Bigl\{(a_1, a_2, \ldots, a_i): \, 1 \leq a_j \leq N_i, j=1,\ldots, i \Bigr\},
\ee
$p_i: A_{i+1} \to A_i$ the projection maps
$p_i(a_1,\ldots,a_{i+1})=(a_1,\ldots,a_{i})$
and
\be\label{eq:A-cntble}
A=\bigsqcup_{i=1}^{\infty }A_i.
\ee
Then there exist countable dense subsets 
\be
{\mathcal{N}}_j = \{x_a^j:\, a\in A\} \subset X_j
\ee
and maps
\be\label{eq:Iji}
I^j_i \colon  A_i \to X_j \quad\textrm{ and }\quad I^j\colon  A\to X_j,
\ee
given by
\be
I^j(a)=I^j_i(a)=x_a^j \in X_j \qquad \forall a \in A_i \subset A, 
\ee
such that 
\be \label{eq:cover}
X_j \subset \bigcup_{a\in A_i} B_{d_j}(x^j_a,\varepsilon_i)
\ee
and 
\be \label{eq:in-0}
I^j_{i+1}(a) \in B_{d_j}(I^j_i(p_i(a)), 2\varepsilon_i)\subset X_j.
\ee
\end{prop}

\begin{thm}
[Timed Gromov's Compactness Theorem, \cite{Che-Perales-Sormani-2025}]\label{thm:timed-Gromov-compactness-CPS}
If $\{(X_j,d_j,\tau_j)\}_{j\in\mathbb{N}}$ is a equibounded and equicompact sequence of compact timed-metric spaces that have a uniform bound on their $1$-Lipschitz functions:
\[
\tau_j(X_j)\subset [0,\tau_{max}],
\]
then a subsequence converges in the intrinsic timed-Hausdorff sense to a compact timed-metric space $(X_\infty, d_\infty, \tau_\infty)$. In fact, 
up to a subsequence, 
there exist distance and time preserving 
timed-Fr\'echet maps,
\[
\varphi_j= \kappa_{\tau_j,\mathcal N_j}\colon  X_j \to [0,\tau_{max}]\times Z \subset \ell^\infty,
\]
where $Z$ is compact, such that 
\be \label{eq:H-A}
d_H^{[0,\tau_{max}]\times Z}(\varphi_j(X_j),\varphi_\infty(X_\infty))\to 0.
\ee
\end{thm}

\bigskip

With the notation introduced in Proposition~\ref{prop:selection}, 
in combination with the following notion, introduced in \cite[Definition~3.3]{Che-Perales-Sormani-2025}, we can get a better description of the Hausdorff convergence above, in terms of uniform convergence. 

\begin{defn}\label{defn:addresses}
Using the notation from Proposition~\ref{prop:selection}, let $\mathcal{A}$ be the set of ordered sequences
\[
\alpha=\{\alpha_1,\alpha_2,\alpha_3,...\} \subset A
\]
that satisfy $\alpha_i \in A_i\subset A$ and 
$p_i(\alpha_{i+1})=\alpha_i$ for all $i \in {\mathbb N}$. We call elements of $\mathcal{A}$ {\bf addresses}.
\end{defn}

Thus, besides \eqref{eq:H-A}, by \cite[Proposition~3.4]{Che-Perales-Sormani-2025} the following holds: 
\be \label{eq:alpha-A}
\lim_{j \to \infty}\sup_{\alpha \in {\mathcal A}}
d_{\ell^\infty}
\Big(\varphi_j({\mathcal I}^j(\alpha)),
\varphi_\infty({\mathcal I}^\infty(\alpha)) \Big) = 0,
\ee
for some ${\mathcal I}^j\colon{\mathcal A}\to X_j$ surjective maps.  
In particular, from  \eqref{eq:alpha-A}, uniform convergence of time and distances,
\be \label{eq:sup-A}
\lim_{j \to \infty}\sup_{\alpha \in {\mathcal A}}
|\tau_j({\mathcal I}^j(\alpha))-
\tau_\infty({\mathcal I}^\infty(\alpha))|= 0,
\ee
\be \label{eq:sup-A-d}
\lim_{j \to \infty}\sup_{\alpha,\alpha' \in {\mathcal A}}
|d_j({\mathcal I}^j(\alpha),{\mathcal I}^j(\alpha'))-
d_\infty({\mathcal I}^\infty(\alpha),{\mathcal I}^\infty(\alpha'))|= 0,
\ee
immediately follow.

\section{Timed-Gromov--Hausdorff distance}\label{sec-timedGH}
Here we prove 
 that  $d_{\tau-H}$
is bi-Lipschitz equivalent to the
$d_{\tau-GH}$.
 We also show that $d_{\tau-GH}$ satisfies the triangle inequality, and establish a compactness theorem for 
 $d_{\tau-GH}$.

\subsection{Distance and time preserving maps into max-product spaces}

Recall the notion of a max-product space from Definition~\ref{def:product space}. The following lemma will be used to establish that $d_{\tau\text{-}GH}$ satisfies the triangle inequality. Moreover, observe that this lemma implies that the class of functions considered in the definition of $d_{\tau\text{-}GH}$ is larger than the class consisting of all functions $(\tau,\zeta): X \to [0,\tau_{\max}] \times Z$ for which $\zeta$ is distance-preserving.

\begin{lem}\label{lem:zeta-dist-pres}
Let $(X,d,\tau)$ be a compact timed-metric space, 
$\tau_{max}= \max \{\tau(x) \, | \, x \in X\}$,
$Z$ a compact metric space
and
\be
\varphi:X\to [0,\tau_{max}]\times Z
\ee
a distance and time preserving map into the
max-product space, $[0,\tau_{max}]\times Z$.
Then $\varphi =(\tau, \zeta)$, where 
$\zeta:X\to Z$ is $1$-Lipschitz. Moreover, $d_Z(\zeta(p), \zeta(q))=d(p,q)$ for any $p,q \in X$ non causally related.
\end{lem}

\begin{proof}
By the definition of a product space,
we know that $\varphi(x)=(\tau_Z(x),\zeta_Z(x))$
where $\tau_Z:X\to [0,\tau_{max}]$
and $\zeta_Z:X\to Z$.  By the definition of $\tau_{prod}$ and since
$\varphi$ is time preserving, we have
\be
\tau_Z(x)=\tau_{prod}(\tau_Z(x),\zeta_Z(x)) =\tau(x).
\ee
Moreover, 
\[
d_Z(\zeta_Z(x),\zeta_Z(y)) \leq d_{prod}(\varphi(x),\varphi(y))=
d_X(x,y),
\]
which proves that $\zeta_Z$ is $1$-Lipschitz.

For the last claim, assume that $p,q\in X$ are not causally related. This is equivalent, by the $1$-Lipschitzness of $\tau$, to
\[
|\tau(p)-\tau(q)| < d(p,q).
\]
Since $d(p,q) = d_{prod}(\varphi(p),\varphi(q))$, it follows that
\[
d(p,q) = d_Z(\zeta_Z(p),\zeta_Z(q)).\qedhere
\]
\end{proof}

\subsection{Equivalence}\label{ss:equivalent notion}

Here we prove that $d_{\tau-H}$
is bi-Lipschitz equivalent to 
$d_{\tau-GH}$.

\begin{thm}\label{thm:timed-H-equiv}
Let $(X_j,d_j,\tau_j)$, $j=1,2$, be
two compact timed-metric spaces,
then we have the following inequalities
between the timed-Hausdorff distance of
Definition~\ref{defn:timed-H} 
and the timed-Gromov--Hausdorff distance of Definition~\ref{defn:t-H-new}:
\be\label{eq:timed-H-equiv-1}
d_{\tau-H}((X_1,d_1,\tau_1),(X_2,d_2,\tau_2))\ge
d_{\tau-GH}((X_1,d_1,\tau_1),(X_2,d_2,\tau_2))
\ee
and
\be\label{eq:timed-H-equiv-2}
d_{\tau-H}((X_1,d_1,\tau_1),(X_2,d_2,\tau_2))\le
2
d_{\tau-GH}((X_1,d_1,\tau_1),(X_2,d_2,\tau_2)).
\ee
\end{thm}

We prove Theorem~\ref{thm:timed-H-equiv} using the following lemma, which establishes the easy direction.

\begin{lem}\label{lem:timed-H-equiv-1}
For any pair of compact timed-metric spaces 
$(X_j,d_j,\tau_j)$, $j=1,2$, the inequality \eqref{eq:timed-H-equiv-1} holds, that is,
\be
d_{\tau-H}((X_1,d_1,\tau_1),(X_2,d_2,\tau_2))\ge
d_{\tau-GH}((X_1,d_1,\tau_1),(X_2,d_2,\tau_2)).
\ee
\end{lem}

\begin{proof}
By Proposition 4.24 and Proposition 4.25 in \cite{SakSor-Notions}, we know that any timed-Fr\'echet map, $\kappa_{X_j,\tau_j}:X_j\to \ell^\infty$,
\be
\kappa_{X_j,\tau_j}(x)=(\tau_j(x),\kappa_{X_j}(x))\in [0,\tau_{j,\max}]\times \ell^\infty
\ee
where $\tau_{j,\max}$ is the max value of $\tau_j$ on $X_j$, is time and distance preserving.  Since the
$X_j$ are compact we have finite
\be
\tau_{max}=\max\{\tau_{1,\max},\tau_{2,\max}\}
\ee
and compact
\be
K=\kappa_{X_1}(X_1)
\cup \kappa_{X_2}(X_2) \subset \ell^\infty.
\ee
Thus, the timed-Fr\'echet maps
\be
\kappa_{X_j,\tau_j}:X_j\to [0,\tau_{max}]\times K \subset \ell^\infty
\ee
in the definition of the timed-Hausdorff distance, as in Definition~\ref{defn:timed-H},
is a subset of the
collection of maps, $\varphi_j$, as in Definition~\ref{defn:t-H-new}. So the infimum in the definition of $d_{\tau-H}$ is larger than the infimum
in $d_{\tau-GH}$,
which gives (\ref{eq:timed-H-equiv-1}).
\end{proof}

\begin{proof}[Proof of Theorem~\ref{thm:timed-H-equiv}]
By Lemma~\ref{lem:timed-H-equiv-1}
we already have (\ref{eq:timed-H-equiv-1})
so we need only to prove
(\ref{eq:timed-H-equiv-2}).
Let
\be\label{eq:D}
D=\inf d_H^{[0,\tau_{max}]\times Z}(\varphi_1(X_1),\varphi_2(X_2)),
\ee
where the infimum is taken over all metric spaces $Z$ and maps $\varphi_j$ as in 
Definition~\ref{defn:t-H-new}.

Fix $\epsilon>0$. Then there exist a compact metric space $Z$ and a pair of time and distance preserving maps,
\be
\varphi^\epsilon_j:(X_j,d_j,\tau_j)\to
([0,\tau_{\max}]\times Z, d_{prod}, \tau_{prod})
\ee
 close to achieving the infimum in
Definition~\ref{defn:t-H-new}:
\be\label{eq-HausEst}
d_H^{[0,\tau_{max}]\times Z}(\varphi^\epsilon_1(X_1),\varphi^\epsilon_2(X_2))
<D
+\epsilon.
\ee

By Proposition~\ref{prop:Hausdorff},
\be
\tilde {\mathcal C}^\epsilon =
\Bigl\{
(\varphi^\epsilon(x_1),\varphi^\epsilon(x_2)) 
 \,|\, d_{Z}(\varphi^\epsilon_1(x_1),
\varphi^\epsilon_2(x_2))<D+\epsilon  \Bigr\}
\ee
is a correspondence between $\varphi_1^\epsilon(X_1)$ and $\varphi_2^\epsilon(X_2)$, with $\dist( \tilde {\mathcal C}^\epsilon) < 2(D+\epsilon)$.
Since $\varphi^\epsilon_j$ are distance preserving, 
\be
{\mathcal C}^\epsilon =
\Bigl\{
(x_1,x_2) 
 \,|\, d_{Z}(\varphi^\epsilon_1(x_1),
\varphi^\epsilon_2(x_2))<D+\epsilon  \Bigr\} \subset X_1 \times X_2 
\ee
is a correspondence between $X_1$ and $X_2$, with $\dist(\mathcal C^\epsilon) < 2(D+\epsilon)$. Hence, by
Lemma~\ref{lem:corr-K} 
there exist
Fr\'echet maps $\kappa_{X_j}: X_j \to \ell^\infty$ such that 
\be \label{eq:L}
d_{\ell^\infty}(\kappa_{X_1}(x_1),\kappa_{X_2}(x_2)) <2(D+\epsilon) 
\ee
for all pairs $(x_1,x_2)\in \mathcal{C}^\epsilon$.
Then by \eqref{eq:L} the timed-Fr\'echet maps
\be
\kappa^\epsilon_{\tau_j,X_j}=
(\tau_j,\kappa_{X_j}): X_j \to [0, \tau_{\max} ] \times \ell^\infty
\ee
satisfy 
\begin{align*}
d_{[0, \tau_{\max}] \times \ell^\infty}(\kappa^\epsilon_{\tau_1, X_1}(x_1),\kappa^\epsilon_{\tau_2, X_2}(x_2)) \leq & \max\{ 
|\tau_1(x_1)- \tau_2(x_2)|,
d_{\ell^\infty}(\kappa_{X_1}(x_1),\kappa_{X_2}(x_2))\} \\
< & \max\{D+\epsilon, 2(D+\epsilon) \}
\end{align*}
for all pairs $(x_1,x_2)\in \mathcal{C}^\epsilon$.
Thus
\be\label{eq:2D}
d_H^{[0, \tau_{\max}] \times \ell^\infty}(\kappa^\epsilon_{\tau_1,X_1}(X_1),
\kappa^\epsilon_{\tau_2,X_2}(X_2))<2(D+\epsilon).
\ee
Taking a sequence of $\epsilon \to 0$ 
we obtain (\ref{eq:timed-H-equiv-2}).
\end{proof}

\subsection{
Triangle inequality}

We now prove that the timed-Gromov--Hausdorff distance $d_{\tau-GH}$, as in Definition~\ref{defn:t-H-new}, satisfies the triangle inequality.  

\begin{thm}\label{thm:triangle}
Given three compact timed-metric spaces $(X_i,d_i,\tau_i)$, $i=1,2,3$, we have the triangle
inequality
\be
\begin{split}
d_{\tau-GH}((X_1,d_1&,\tau_1), (X_3,d_3,\tau_3))
\le\\
&d_{\tau-GH}((X_1,d_1,\tau_1),(X_2,d_2,\tau_2))
+
d_{\tau-GH}((X_2,d_2,\tau_2),(X_3,d_3,\tau_3)).
\end{split}
\ee
\end{thm}

\begin{proof}
We will abuse notation and only write $X_i$ instead of the tripe $(X_i,d_i, \tau_i)$.
Let 
\be\label{eq-distij}
D_{12}=d_{\tau-GH}(X_1, X_2) \qquad \text{and}\qquad 
D_{23}=d_{\tau-GH}(X_2, X_3),
\ee
$\tau_{\max 12}=\max \{ \max_{x\in X_1}\tau_1(x), \max_{x \in X_2}\tau_2(x)\}$, 
and similarly define $\tau_{\max 23}$ and $\tau_{\max 13}$. 

Fix $\epsilon>0$. Then there exist compact metric spaces $Z_{12}, Z_{23}$ and pairs of time and distance preserving maps,
\be
\varphi_j:(X_j,d_j,\tau_j)\to
([0,\tau_{\max 12}]\times Z_{12}, d_{12}, \tau_{12}) \qquad j=1,2,
\ee
and 
\be
\phi_j:(X_j,d_j,\tau_j)\to
([0,\tau_{\max 23}]\times Z_{23}, d_{23}, \tau_{23}) \qquad j=2,3,
\ee  
such that
\be
d_H^{[0,\tau_{\max 12}]\times Z_{12}}(\varphi_1(X_1),\varphi_2(X_2))
<D_{12}
+\epsilon.
\ee
and
\be
d_H^{[0,\tau_{\max 23}]\times Z_{23}}(\phi_2(X_2),\phi_3(X_3))
<D_{23}+\epsilon.
\ee
By Lemma 
\ref{lem:zeta-dist-pres}, we know that $\varphi_j=(\tau_j, \xi_j)$, $j=1,2$, 
and $\phi_j=(\tau_j, \chi_j)$, $j=2,3$. 
We construct a compact metric space $Z$ by gluing $Z_{12}$ and $Z_{23}$ along 
$X_2$:
\be
Z= Z_{12} \disjointunion Z_{23}/ \sim, \qquad \text{ where }\ z \sim z' \iff  {\exists x\in X_2 \,s.t.\, z=\xi_2(x),\, z'=\chi_2(x)},
\ee
with metric given by
\begin{align}
d_Z(z,z')= 
\begin{cases}
   d_{Z_{12}}(z, z') & \quad z,z' \in Z_{12}\\
    d_{Z_{23}}(z, z') & \quad z,z' \in Z_{23}\\
   \inf_{x \in X_2}\{ d_{12}((0,z), \varphi_2(x)) + d_{23}(\phi_2(x),(0,z'))\} & \quad z \in Z_{12}, \, z'\in Z_{23}\\
d_Z (z',z) & \quad z' \in Z_{12}, \, z\in Z_{23}.
\end{cases}
\end{align}
The fact that $d_Z$ is symmetric and non-negative is immediate from its definition. To check $d_Z$ is definite, the only interesting case is when $z\in Z_{12}$ and $z'\in Z_{23}$ are such that $d_Z(z,z')=0$, which by compactness of $X_2$ and the definition of $d_Z$ implies there is some $x\in X_2$ such that
\[
d_{12}((0,z), \varphi_2(x)) + d_{23}(\phi_2(x),(0,z')) = 0,
\]
thus
\[
d_{Z_{12}}(z,\xi_2(x)) = d_{Z_{23}}(\chi_2(x),z') = 0, 
\]
i.e. $z\sim z'$. For the triangle inequality, consider the case $z,z'\in Z_{12}$ and $z''\in Z_{23}$. Then, for any $x,x'\in X_2$ we have:
\begin{align*}
d_Z(z,z') &= d_{12}((0,z),(0,z')) \\
&\leq d_{12}((0,z),\varphi_2(x)) + d_{12}(\varphi_2(x),\varphi_2(x')) + d_{12}(\varphi_2(x'),(0,z'))\\
&= d_{12}((0,z),\varphi_2(x)) + d_{23}(\phi_2(x),\phi_2(x')) + d_{12}(\varphi_2(x'),(0,z'))\\
&\leq d_{12}((0,z),\varphi_2(x)) + d_{23}(\phi_2(x),(0,z''))+d_{23}((0,z''),\phi_2(x')) + d_{12}(\varphi_2(x'),(0,z'))
\end{align*}
and by taking the infimum over $x,x'\in X_2$, we get
\[
d_Z(z,z') \leq d_Z(z,z'')+d_Z(z'',z').
\]
On the other hand
\begin{align*}
d_Z(z,z'') &\leq  d_{12}((0,z),\varphi_2(x)) + d_{23}(\phi_2(x),(0,z'')) \\
&\leq d_{12}((0,z),(0,z')) + d_{12}((0,z'),\varphi_2(x)) + d_{23}(\phi_2(x),(0,z'')) \\
&= d_Z(z,z') + d_{12}((0,z'),\varphi_2(x)) + d_{23}(\phi_2(x),(0,z'')), 
\end{align*}
and by taking the infimum over $x\in X_2$, we get
\[
d_Z(z,z'') \leq d_Z(z,z')+d_Z(z',z'').
\]
The rest of the cases are analogous or easier.

The maps
\be
\psi_1=(\tau_1, \iota_{12} \circ \xi_1): X_1 \to [0, \tau_{\max 13}]\times Z 
\ee
\be
\psi_3=(\tau_3, \iota_{23} \circ\chi_3): X_3 \to [0, \tau_{\max 13}]\times Z,  
\ee
where $\iota_{ij}: Z_{ij} \to Z$ are the canonical inclusion maps, are clearly distance and time preserving.  We claim that
\be
d_{H}^{[0, \tau_{\max 13}]\times Z} ( \psi_1(X_1), \psi_3(X_3)) < D_{12}+ D_{23}+ 2\epsilon.
\ee
Since $\epsilon>0$ was taken arbitrarily, this will establish the triangle inequality.

By proposition~\ref{prop:Hausdorff}  
there are correspondences
\be
\mathcal C_{12} \subset X_1\times X_2 \qquad \text{and}\qquad \mathcal C_{23} \subset X_2\times X_3 
\ee
such that 
for all $(x_1, x_2) \in \mathcal{C}_{12}$
we have
\be
d_{prod}(\varphi_1(x_1),
\varphi_2(x_2))<D_{12}+\epsilon.
\ee
Similarly for ${\mathcal C_{23}}$. 
Let
\be
{\mathcal C_{13}} \subset X_1\times X_3
\ee
be the composition of $\mathcal{C}_{12}$ and $\mathcal{C}_{23}$ (see Definition~\ref{defn:composition of correspondences}). 
This correspondence gives the claimed estimate.  Indeed, for any $(x_1,x_3)\in \mathcal{C}_{13}$, let $x_2 \in X_2$ such that 
$(x_1, x_2) \in {\mathcal C_{12}}$ and 
$(x_2, x_3) \in {\mathcal C_{23}}$, then 
\begin{align*}
d_{prod}(\varphi_1(x_1),\varphi_2(x_2))<D_{12}+\epsilon\\
d_{prod}(\phi_2(x_2),\phi_3(x_3))<D_{23}+\epsilon,
\end{align*}
which implies
\begin{align}
d_{Z_{12}}(\xi_1(x_1),\xi_2(x_2)) < D_{12}+\epsilon, \label{eq:trianglineq1}\\
d_{Z_{23}}(\chi_2(x_2),\chi_3(x_3)) < D_{23}+\epsilon \label{eq:trianglineq2}
\end{align}
and
\begin{align}
|\tau_1(x_1)-\tau_2(x_2)| < D_{12}+\epsilon, \label{eq:trianglineq3}\\
|\tau_2(x_2)-\tau_3(x_3)| < D_{23}+\epsilon. \label{eq:trianglineq4}
\end{align}
By definition of $d_Z$, \eqref{eq:trianglineq1}, and \eqref{eq:trianglineq2},
\[
d_Z(\iota_{12}\circ\xi_1(x_1),\iota_{23}\circ\chi_3(x_3)) \leq D_{12} + D_{23} + 2\epsilon
\]
and by \eqref{eq:trianglineq3} and \eqref{eq:trianglineq4},
\[
|\tau_1(x_1)-\tau_3(x_3)| < D_{12} + D_{23} + 2\epsilon.
\]
It follows that for all $(x_1, x_3) \in \mathcal{C}_{13}$
\begin{align*}
d_{prod}(\psi_1(x_1),\psi_3(x_3)) &= \max\{|\tau_1(x_1)-\tau_3(x_3)|,d_Z(\iota_{12}\circ \xi_1(x_1),\iota_{23}\circ \xi_3(x_3))\}\\
&< D_{12} + D_{23} + 2\epsilon.
\end{align*}
Since $\mathcal{C}_{13}$ is a correspondence, this implies \be
d_{H}^{[0, \tau_{\max 13}]\times Z} ( \psi_1(X_1), \psi_3(X_3)) < D_{12}+ D_{23}+ 2\epsilon.
\ee
\end{proof}

\subsection{{$\tau$-GH} compactness theorem}

As an application of Gromov's Compactness theorem, c.f. 
Theorem \ref{thm-GHcriterion}, we now prove a timed-Gromov--Hausdorff compactness result.

\begin{proof}[Proof of Theorem \ref{thm:timed-GH-compactness}]

Following the notation of Theorem  \ref{thm-GHcriterion} and its proof, we  define
\be
\varphi_j= (\tau_j(x), h_j(x)): X_j \to [0, \tau_{\max}] \times F
\ee
and endow $[0, \tau_{\max}] \times F$ with the max-product structure. 
Since $\tau_j$ are $1$-Lipschitz and $h_j$ are distance preserving, we conclude that
\be
d_{prod}(\varphi_j(x),\varphi_j(y))= 
\max\{ |\tau_j(x) - \tau_j(y)|, \, \|h_j(x)) - h_j(x))\| \}= d_j(x,y).
\ee
Hence, $\varphi_j$ are time and distance preserving. 

\medskip

Since $([0, \tau_{\max}] \times F, d_{prod})$ is compact, 
by Blaschke's compactness theorem (c.f. Chapter 7 in \cite{BBI}), there exist a compact space
$Y' \subset [0, \tau_{\max}] \times F$ 
and a subsequence, that we do not relabel, such that 
\begin{equation}\label{eq-tauHconv}
   d_H^{ [0, \tau_{\max}] \times F} (\varphi_j(X_j), Y')\to 0. 
\end{equation}

We now define 
$\varphi_\infty: Y_\infty \to [0,\tau_{max}]\times F$ and 
$\tau_\infty: Y_\infty \to \R$ in the following way:

For $f\in Y_\infty$, by the Hausdorff convergence of the sets 
$h_j(X_j)$, \eqref{eq-GHcriterion-H},
there exists a sequence of 
$x_j \in X_j$ such that 
\be
\|h_j(x_j)-f\| \to 0.
\ee
Now by the Hausdorff convergence of the 
$\varphi_j(X_j)$, \eqref{eq-tauHconv},
there exists a further subsequence of $x_j \in X_j$, that we do not relabel, and unique $t \in [0,\tau_{\max}]$ such that 
$(t,f)\in Y'$ and
\be
d_{prod}(\varphi_j(x_j),(t,f)) \to 0.
\ee
We define
\be 
\varphi_\infty(f) :=(t,f) \quad \text{and} \quad \tau_\infty(f)=t. 
\ee

We finally claim that 
$(Y_\infty, \| \,\cdot\, \|, \tau_\infty)$ is 
a compact timed-metric space which is (up to a subsequence) the timed-Gromov--Hausdorff limit of 
$(X_j, d_j, \tau_j)$. 
We prove this by showing that $\tau_\infty$ is $1$-Lipschitz, $\varphi_\infty$ is distance preserving (the fact that it is time preserving is obvious from its definition), and $Y'=\varphi_\infty(Y_\infty)$. 

Indeed, take $f,f' \in Y_\infty$. 
Then there exist $x_j, x_j' \in X_j$ such that 
\be
 \| h(x_j)-f\| , \ \| h(x_j')- f'\| \to 0
\ee
and 
\be
d_{prod}(\varphi_j(x_j), \varphi_{\infty}(f)), \ 
d_{prod}(\varphi_j(x_j'), \varphi_\infty(f'))\to 0.
\ee
Then, by the definition of the prod metric,
\be
    |\tau_\infty(f)-\tau_\infty(f')|\leq d_{prod}(\varphi_\infty(f), \varphi_\infty(f')),
\ee
and
\begin{align}
d_{prod}(\varphi_\infty(f), \varphi_\infty(f'))= &
\lim_{j \to \infty}d_{prod}(\varphi_j(x_j), \varphi_j(x'_j)) \\
= & \lim_{j\to \infty} d_j(x_j, x'_j) \\
= & \lim_{j\to \infty} \| h(x_j)-h(x_j')\|  \\
= & d_F(f,f').
\end{align}
Thus, $\varphi_\infty$ is distance preserving and $\tau_\infty$ is $1$-Lipschitz. 

Finally, if $(t,f)\in Y'$ then, by \eqref{eq-tauHconv}, there is a sequence $x_j\in X_j$ such that
\[
\varphi_j(x_j) = (\tau_j(x_j),h_j(x_j))\to (t,f),
\]
which implies $\|h_j(x_j)-f\|\to 0$ and $\tau_j(x_j)\to t$. Therefore, $f\in Y_\infty$ and $\varphi_\infty(f) = (t,f)$, which proves $Y'=\varphi_\infty(Y_\infty)$.
\end{proof}

\begin{rmrk}\label{rmrk-converseEquibddEquicpct}
If $(X_j,d_j,\tau_j)$ are a sequence of compact timed-metric spaces that converges 
in the intrinsic timed-Hausdorff sense or the timed-Gromov--Hausdorff sense to a compact timed-metric space,
$(X_\infty,d_\infty,\tau_\infty)$, 
then from the standard theory of Gromov--Hausdorff convergence, 
(see, for example, Chapter 7 in \cite{BBI}, in particular, Exercises 7.3.13, 7.3.14, and Proposition 7.4.12), we know that $(X_j,d_j)$ are equibounded 
and
equicompact. Moreover, using the notion of intrinsic timed-Hausdorff or the timed-Gromov--Hausdorff distance, there exists $\tau_{\max}>0$ such that 
$\bigcup_{j\in \mathbb N}\tau_j(X_j) \subset [0, \tau_{\max}]$. 
\end{rmrk}

\section{Understanding causality}\label{s:understanding causality}

In this section, we show that the causal structure for a timed-metric space introduced in Definition \ref{defn:causal-structure} defines a partial order, and study when this structure is preserved in the limit. 
\medskip

We start with a result that formalizes the fact that the causal structure for timed-metric spaces is a generalization of the standard causal structure of space-times.

\begin{prop}
Given a timed-metric space $(X,d,\tau)$, the induced causal structure defines a partial order.
\end{prop}

\begin{proof}
For any $p\in X$ it is clear that $p\in J^+(p)$, since 
\[
\tau(p) - \tau(p) = 0 = d(p,p).
\]
On the other hand, if $p\in J^+(q)$ and $q\in J^+(p)$ then 
\[
\tau(p)-\tau(q) = d(q,p) = d(p,q) = \tau(q)-\tau(p),
\]
which implies $d(p,q) = 0$, i.e.\ $p=q$. Finally, if $p \in J^+(q)$ and $q\in J^+(r)$,  then
\[
\tau(p)-\tau(q) = d(q,p) \quad \text{and}\quad 
\tau(q)-\tau(r) = d(r,q)
\]
which by the triangle inequality yields
\[
\tau(p)-\tau(r) = d(q,p)+d(r,q) \geq d(r,p)
\]
Hence, $\tau(p)-\tau(r) \geq d(r,p)$. The reversed inequality 
follows from the fact that $\tau$ is $1$-Lipschitz. Hence, $p \in J^+(r)$. This concludes the proof of the proposition.
\end{proof}

\begin{example}
Consider the timed-metric space $(X,d,\tau)$ where $\tau = d(\cdot,p)$ for some $p \in X$. In this example, the causal relation $y\in J^+(x)$ is equivalent to $p$, $x$, and $y$ being aligned, i.e. 
\be
d(y,p)=d(y,x)+d(x,p).
\ee
In particular, if $X$ is non-branching and geodesic, the set $J^+(x)$, for $x\neq p$, is the image of the maximal geodesic starting at $x$ that can be extended to $p$, whereas $J^+(p) = X$.    
\end{example}

\begin{rmrk}
Observe that we could define chronological futures and pasts as
\[
I^\pm(p) : = \interior(J^\pm(p)). 
\]
With this definition, the relation $q\in I^-(p)$ is transitive, and therefore, it could be used to define a chronological structure in $X$, which
agrees with the canonical chronological structure for Lipschitz space-times,
but for lower regularity, such as $C^0$, 
the chronological structures may not coincide due to the existence of bubbling (see \cite{Ling-Aspects}).    
\end{rmrk}

Since we have a causal structure on timed-metric spaces, we can define causal curves.
\begin{defn}
Given a timed-metric space $(X,d,\tau)$, 
we say that a non-constant continuous curve, $\gamma\colon I\to X$, is a \textbf{future-directed causal curve} if for any $s\leq t$ in $I$, $\gamma(s) \in J^-(\gamma(t))$. Analogously, we define
past-directed causal curves. Causal curves are curves that are either future-directed or past-directed. 
\end{defn}

The following two propositions are respectively analogous to \cite[Corollary~3.19]{SV-Null} and to part of \cite[Lemma~3.20]{SV-Null}.

\begin{prop}\label{prop:causal curves are realizers}
Let $(X,d,\tau)$ be a timed-metric space. Then any causal curve of $X$ is a distance realizer between any of its points. 
\end{prop}

We recall that for a metric space $(X,d)$ and points $x,y \in X$, a continuous curve $\gamma : [s,t] \to X$ from 
$x$ to $y$ is a distance realizer if 
\[
d(x,y)= \sup_{s=s_0<\dots < s_N =t} \sum_{i=1}^{N} d(\gamma(s_{i-1}),\gamma(s_i)).
\]

\begin{proof}
Let $\gamma\colon I\to X$ be a causal curve, and without loss of generality, assume that $\gamma$ is future-directed. Let $s<t$ in $I$. Then for any $s= s_0 < \dots < s_N = t$, 
\[
\gamma(s_{i-1}) \in J^-(\gamma(s_i))
\]
for $i=1,\dots, N$. Therefore,  
\begin{align*}
\sum_{i=1}^{N} d(\gamma(s_{i-1}),\gamma(s_i)) &= \sum_{i=1}^{N} \tau(\gamma(s_i))-\tau(\gamma(s_{i-1})) \\
&= \tau(\gamma(t))-\tau(\gamma(s)) = d(\gamma(s),\gamma(t)).
\end{align*}
Taking the supremum over all possible partitions of 
$[s,t]$, we get that 
$\gamma$ is a distance realizer between $\gamma(s)$ and $\gamma(t)$.
\end{proof}

\begin{prop}\label{prop:piecewise causal realizers}
Let $(X,d,\tau)$ be a timed-metric space. Let $\gamma\colon [a,b]\to X$ be a piecewise causal curve, i.e., there exists a partition $a=s_0<\dots <s_N=b$ such that $\gamma|_{[s_{i-1},s_i]}$ is causal for $i=1,\dots, N$. If $\gamma$ is a distance realizer between any of its points then either $\gamma$ is causal or for some subpartition $a = s_0' <\dots s'_{N'} =b$ of $a=s_0<\dots <s_N=b$ the following holds: $\gamma|_{[s_{i-1}',s_i']}$ and $\gamma|_{[s_i',s_{i+1}']}$ have opposite causal orientation and no point in $\gamma|_{[s_{i-1}',s_i')}$ is causally related to points in $\gamma|_{(s_{i}',s_{i+1}']}$, for $i=1,\dots, N-1$.
\end{prop}

\begin{proof}
Assume that $\gamma$ is not causal. Thus, without loss of generality, we can assume that for some $s_i$, $\gamma|_{[s_{i-1},s_i]}$ is future-directed causal and $\gamma|_{[s_i,s_{i+1}]}$ is past-directed causal. Let $u\in [s_{i-1},s_i)$ and $v\in (s_i,s_{i+1}]$. If $\gamma(u) \in J^-(\gamma(v))$ then 
\begin{align*}
d(\gamma(u),\gamma(v)) &= \tau(\gamma(v))-\tau(\gamma(u)) \\
&< \tau(\gamma(s_i))-\tau(\gamma(u))+\tau(\gamma(s_i))-\tau(\gamma(v)) \\
&= d(\gamma(u),\gamma(v)),
\end{align*}
where the strict inequality follows from the fact that $\gamma|_{[s_i,s_{i+1}]}$ is non-constant and past-directed, therefore $\tau(\gamma(s_i))-\tau(\gamma(v))=d(\gamma(v),\gamma(s_i))>0$. This is a contradiction, and we get an analogous contradiction assuming $\gamma(u)\in J^+(\gamma(v))$. Therefore, $\gamma(u)$ and $\gamma(v)$ are not causally related. By letting $a=s'_0<\dots<s'_{N'}=b$ be the ordered values of $s_i$ such that $\gamma$ changes causal orientation, the result follows.
\end{proof}

\subsection{Causality and convergence}

We now apply Theorem \ref{thm:timed-GH-compactness} and Theorem~\ref{thm:timed-Gromov-compactness-CPS}
to establish some results describing how causality behaves under intrinsic timed-Hausdorff and timed-Gromov--Hausdorff convergence.

\begin{proof}[Proof of Corollary~\ref{cor:CausalSeq}]
By Remark \ref{rmrk-converseEquibddEquicpct} and Theorem \ref{thm:timed-GH-compactness}, there exist
  $\varphi_j: X_j \to [0,\tau_{max}]\times Z$, $j \in \mathbb N$, and $\varphi_\infty: X_\infty \to [0,\tau_{max}]\times Z$ isometric embeddings that realize the $\tau$-GH convergence.

We now prove (\ref{item:CausalSeq1}).
Then, the subsequences $\varphi_j(p_j)$ and $\varphi_j(q_j)$ converge in $Z_{prod}=[0,\tau_{max}]\times Z$ to some $z_\infty, z'_\infty \in \varphi_\infty(X_\infty)$, respectively. 
Thus, $z_\infty=\varphi_\infty(p_\infty)$ and 
 $z'_\infty=\varphi_\infty(q_\infty)$ for some 
 $p_\infty, q_\infty \in X_\infty$. It remains to show that the points satisfy the stated causal relations.

Recall that $p_j\in J^+_{X_j}(q_j)$ iff 
\be
d_j(p_j, q_j)= \tau_j(p_j) - \tau_j(q_j).
\ee
Since the $\varphi_j$ are time and distance preserving, this can be rewritten as 
\be
d_{prod}(\varphi_j(p_j),\varphi_j(q_j))= \tau_{prod}(\varphi_j(p_j)) - \tau_{prod}(\varphi_j(q_j)),
\ee
which by the $\tau$-GH convergence,
converges to 
\be
d_{prod}(\varphi_\infty(p_\infty),\varphi_\infty(q_\infty))= \tau_{prod}(\varphi_\infty(p_\infty)) - \tau_{prod}(\varphi_\infty(q_\infty)).
\ee
This proves (\ref{item:CausalSeq1}).

Now we show (\ref{item:CausalSeq2}). Let $p_\infty,p_\infty'\in X_\infty$ such that $p_\infty \not\in J^+_{X_\infty}(p_\infty')$. By Corollary~\ref{cor:pointtimeConvergence}, there exist $p_j,p_j'\in X_j$ such that $\varphi_j(p_j)\to p_\infty$ and $\varphi_j(p_j')\to \varphi_\infty(p_\infty')$. If the statement is false, then there are common subsequences $p_{j_k}$, $p_{j_k}'$ such that $p_{j_k}\in J^+_{X_{j_k}}(p_{j_k}')$. By Item (\ref{item:CausalSeq1}), it follows that $p_\infty \in J^+_{X_\infty}(p_\infty')$, which is a contradiction.
\end{proof}

For the sake of completeness, we rewrite and prove both second items of Corollary \ref{cor:pointtimeConvergence} and Corollary \ref{cor:CausalSeq}, for the intrinsic timed-Hausdorff distance using addresses in the sense of Definition~\ref{defn:addresses}. Note that it is precisely in these cases where the converging sequence is automatically given by the properties of the addresses. 

\begin{cor}\label{cor:CausalSeqWithAddressesItem2}
Suppose a sequence of compact 
timed-metric spaces converges 
in the inrinsic timed-Hausdorff sense to a compact timed-metric space,
\be
(X_j,d_j,\tau_j)\tHto
(X_\infty,d_\infty,\tau_\infty).
\ee
Then for a subsequence $j_k$
the conclusions of Theorem~\ref{thm:timed-Gromov-compactness-CPS}
hold and the following is satisfied:

\medskip

Given $p_\infty=\mathcal{I}^{\infty}(\alpha) \in X_\infty$, the sequence
$p_j=\mathcal{I}^{j}(\alpha)\in X_j$ converges to $p_\infty$:
\[
\varphi_j( \mathcal{I}^{j}(\alpha)  )\to \varphi_\infty(\mathcal{I}^{\infty}(\alpha)).
\]
Moreover, given $\alpha,\alpha'\in \mathcal{A}$ such that 
\be
\mathcal{I}^{\infty}(\alpha) \notin J^+_{X_{\infty}}(\mathcal{I}^{\infty}(\alpha')),
\ee
for $j$ sufficiently large it holds, 
\be
\mathcal{I}^j(\alpha) \notin J^+_{X_j}(\mathcal{I}^j(\alpha')).
\ee
\end{cor}

\begin{proof}
By Remark \ref{rmrk-converseEquibddEquicpct}, 
the conclusions of Theorem \ref{thm:timed-Gromov-compactness-CPS}. Now, the first claim follows by the uniform convergence of 
$\varphi_j \circ \mathcal I^j: \mathcal A \to Z$ to $\varphi_\infty \circ \mathcal I^\infty: \mathcal A \to Z$ given in \eqref{eq:alpha-A}. To establish the second claim, let $p_j=\mathcal{I}^j(\alpha),p_j'=\mathcal{I}^j(\alpha')\in X_j$ for $j \in \mathbb N \cup \{\infty\}$. 
Assume that the corollary is false, 
then there exists a subsequence of $j_k$ such that 
\be
p_{j_k}\in J^+_{X_{j_k}}(p'_{j_k}).
\ee
Since $p_\infty\notin J^+_{X_\infty}(p'_\infty)
$, we have by Definition~\ref{defn:timed-metric-space} that
\be
\tau_\infty(p_\infty)-
\tau_\infty(p'_\infty)\neq
d_\infty(p_\infty,p'_\infty).
\ee
So there exists $\epsilon>0$ such that
\be\label{eq:contrthis}
|\tau_\infty(p_\infty)-
\tau_\infty(p'_\infty)-
d_\infty(p_\infty,p'_\infty)|
>\epsilon.
\ee
Once again by \eqref{eq:alpha-A}, there exists
$j$ sufficiently large for which we have
\be \label{eq:alpha-A-used}
\sup_{\alpha \in \mathcal{A}}
d_{\ell^\infty}\Big(\varphi_{j}({\mathcal I}^{j}(\alpha)),\varphi_\infty({\mathcal I}^\infty(\alpha))\Big) <\epsilon/3.
\ee
Thus, 
\begin{eqnarray}
&|\tau_j(p_j)-\tau_\infty(p_\infty)| <
\epsilon/12\\
&|\tau_j(p'_j)-\tau_\infty(p'_\infty)| < \epsilon/12\\
&\sup_{a\in A} |d_j(p_j,x^j_a)-d_\infty(p_\infty,x^\infty_a) | < \epsilon/12\\
 &\sup_{a'\in A}|d_j(p'_j,x^j_{a'})-d_\infty(p'_\infty,x^\infty_{a'})|<\epsilon/12.
\end{eqnarray}
Choosing first $a\in A$ 
so that
\be
d_\infty(p_\infty,x^\infty_{a})<\epsilon/12
\ee
and
$a'\in  A$ so that
\be
d_\infty(p'_\infty,x^\infty_{a'})<\epsilon/12
\ee
we have by the triangle inequality,
\begin{align}
|d_\infty(p_\infty,p'_\infty)
-d_j(p_j,p'_j)| &\leq 
\epsilon/12+\epsilon/12+\epsilon/12+\epsilon/12 \\
 &=  \epsilon/3.
\end{align}
So for $j=j_k$ sufficiently large 
\begin{align}
|\tau_\infty(p_\infty)-
\tau_\infty(p'_\infty)-
d_\infty(p_\infty,p'_\infty)|
&\leq 
|\tau_j(p_j)-
\tau_j(p'_j)-
d_j(p_j,p'_j)|
 + 
|\tau_j(p_j)-\tau_\infty(p_\infty)| \nonumber\\
&\quad + |\tau_j(p'_j)-\tau_\infty(p'_\infty)|
+ |d_\infty(p_\infty,p'_\infty)
-d_j(p_j,p'_j)| \\
 &\leq \epsilon/12+\epsilon/12+\epsilon/3=\epsilon/2
 \end{align}
 which contradicts
 (\ref{eq:contrthis}).
\end{proof}

\section{Causally-null timed-metric spaces}\label{s:causally-null tms}

We now introduce the null distance on a timed-metric space, based on the notion of null distance on a Lorentzian manifold introduced by Sormani--Vega
\cite{SV-Null}, and also introduce the notion of causally-null timed-metric space based on the smooth notion given by Sakovich--Sormani \cite{SakSor-Notions}.

\begin{defn}\label{defn:null-distance}
Given a timed-metric space $(X,d,\tau)$, 
we define $\hat{d}_{d,\tau}: X \times X \to [0, \infty]$
as
\be\label{eq:null}
\hat{d}_{d,\tau}(p,q)=\inf \sum_{i=1}^{N} |\tau(p_i)-\tau(p_{i-1})| 
\ee
where the infimum is over all collections of points, if they exist, $\{p_i\}_{i=0}^{N}$, where  
$p=p_{0}$ and $p_{N}=q$, and $p_i$, $p_{i-1}$ are causally related, possibly alternating past
with future:
\be\label{eq:null-points}
p_i\in J^-(p_{i-1}) \quad \text{ or } \quad p_i \in J^+(p_{i-1}), \quad 
 i=1,\ldots,N.
\ee
Otherwise, $\hat{d}_{d,\tau}(p,q)=\infty$. 
We call $\hat{d}_{d,\tau}$ the  associated \textbf{null distance} of 
$(X,d,\tau)$, although it is an extended distance in general. If $\hat{d}_{d,\tau}=d$ we say that $(X,d,\tau)$ is \textbf{causally null}.  
\end{defn}

\bigskip

\begin{proof}[Proof of Proposition \ref{prop-null-distance is greater}]

In what follows, to avoid confusion, let us denote  $\hat{d}_{d,\tau}$ as $\hat{d}$, and the causal structure of $(X,\hat d,\tau)$ as $J^\pm_{\hat d}$. 

First we prove that $\hat{d}$ is an extended distance. It is clear that $\hat{d}$ is non-negative and symmetric. Let $x,y,z\in X$ and without loss of generality assume that 
$\hat{d}(x,y), \hat{d}(y,z)< \infty$.
Then for any $\varepsilon>0$ let $\{p_i\}_{i=0}^{N}$ and $\{q_i\}_{i=0}^{M}$ be such that $p_0=x$, $p_N = y = q_0$, $q_M=z$, and $p_i\in J^\pm(p_{i-1})\cap J^\pm(p_{i+1})$ and $q_i\in J^\pm(q_{i-1})\cap J^\pm(q_{i+1})$, and 
\be
\sum_{i=1}^{N} |\tau(p_i)-\tau(p_{i-1})| < \hat{d}(x,y) + \varepsilon/2 
\ee
and
\be
\sum_{i=1}^{M} |\tau(q_i)-\tau(q_{i-1})| < \hat{d}(y,z) + \varepsilon/2.
\ee
Then
\be
\hat{d}(x,z) \leq \sum_{i=1}^{N} |\tau(p_i)-\tau(p_{i-1})| +  \sum_{i=1}^{M} |\tau(q_i)-\tau(q_{i-1})| \leq \hat{d}(x,y) + \hat{d}(y,z) + \varepsilon.
\ee
By letting $\varepsilon\to 0$, it follows that $\hat d$ satisfies the triangle inequality.

Now take $x,y \in X$. By definition,
\begin{align*}
\hat{d}(x,y) = \inf\left\{\sum_{i=1}^{N} |\tau(p_i)-\tau(p_{i-1})|: x = p_0,\dots, p_{N} = y,\ p_{i}\in J^{\pm}(p_{i-1})\right\}.
\end{align*}
Since $p\in J^{+/-}(q)$ is equivalent to $|\tau(q)-\tau(p)| = d(q,p)$, we get
\be
\sum_{i=1}^{N} |\tau(p_i)-\tau(p_{i-1})| = \sum_{i=1}^{N} d(p_i,p_{i-1}) \geq d(x,y)
\ee
whenever $x = p_0,\dots, p_{N} = y$, $p_{i}\in J^\pm_{d,\tau}(p_{i-1})$. Taking the infimum over such tuples $\{p_i\}_{i=0}^{N}$, inequality \eqref{eq:null-distance is greater} follows. In particular, \eqref{eq:null-distance is greater} implies that $\hat{d}$ is positive definite, and 
that $\tau$ is $1$-Lipschitz with respect to $\hat{d}$, given that $\tau$ is $1$-Lipschitz with respect to $d$.

Now assume that $\hat{d}(x,y)<\infty$ for any $x,y\in X$. Therefore $(X,\hat{d},\tau)$ is a timed-metric space. To prove that $\hat{d}_{\hat{d},\tau}= \hat d$ in this case, we first show that the causal structures agree:
\begin{equation}\label{eq-structuresAgree}
    J^-(y)=J^-_{\hat{d}}(y) \quad \forall y \in X. 
\end{equation}
Indeed, let $x\in J^-_{d}(y)$. Taking $p_0 = p_1 = x$ and $p_2 = y$ in the definition of $\hat{d}(x,y)$ yields 
\begin{align*}
\hat{d}(x,y) \leq & (\tau(x)-\tau(x)) + (\tau(y) - \tau(x)) \\
= & d(x,y) \leq \hat{d}(x,y), 
\end{align*}
where in the last part we used \eqref{eq:null-distance is greater}.
Hence, all inequalities are equalities. In particular,  $\hat{d}(x,y) = \tau(y)-\tau(x)$, which proves $x\in J^-_{\hat{d}}(y)$.  Conversely, if $x\in J^-_{\hat{d}}(y)$, by \eqref{eq:null-distance is greater} and $\tau$  $1$-Lipschitz with respect to $d$, we get
\begin{align*}
\tau(y) - \tau(x) = & \hat{d}(x,y)  
\geq  d(x,y) \geq \tau(y)-\tau(x).
\end{align*}
Therefore, $d(x,y) = \tau(y)-\tau(x)$, i.e., 
$x\in J^-(y)$. This concludes the proof of \eqref{eq-structuresAgree}. Thus, 
\begin{align*}
\hat{d}_{\hat{d}, \tau}(x,y) 
&= \inf\left\{\sum_{i=1}^{N} |\tau(p_i)-\tau(p_{i-1})| : x=p_0,\dots, p_{N}=y,\ p_{i}\in J^\pm_{\hat{d}}(p_{i-1})\right\}\\
&= \inf\left\{\sum_{i=1}^{N} |\tau(p_i)-\tau(p_{i-1})| : x=p_0,\dots, p_{N}=y,\ p_{i}\in J^\pm(p_{i-1})\right\}\\
&= \hat{d}(x,y).
\end{align*}
\end{proof}

We now adapt the notion of sufficiently causally connected, introduced 
by Kunzinger-Steinbauer \cite{Kunzinger-Steinbauer-22}, which provides a sufficient condition for obtaining a causally-null timed-metric space.

\begin{defn}
Given a timed-metric space $(X,d,\tau)$, 
we say that $(X,d,\tau)$ is \textbf{causally connected} if for any $p,q\in X$ such that $p\in J^-(q)$ there exists a future-directed causal curve from $p$ to $q$. Furthermore, we say that $(X,d,\tau)$ is \textbf{sufficiently causally connected} if it is path-connected, causally connected and $X = \bigcup_{x\in X} \interior(J^-(x))\cup \interior(J^+(x))$.
\end{defn}

An argument analogous to the proofs of \cite[Lemma~3.5]{Kunzinger-Steinbauer-22} and \cite[Lemma~3.5]{SV-Null} yields the following result.

\begin{prop}\label{prop:scc imply causally null}
Let $(X,d,\tau)$ be a sufficiently causally connected timed-metric space. Then $\hat{d}_{d,\tau}(p,q)<\infty$ for all $p,q\in X$, and thus, $(X,\hat{d}_{d,\tau},\tau)$ is a causally-null timed-metric space. 
\end{prop} 

\begin{proof}
Let $p, q \in X$. We will show that 
there is a piecewise causal curve that connects $p$ to $q$.  Since $X$ is path-connected, there is a curve 
$\alpha: I \to X$ from $p$ to $q$. 
By compactness of $\alpha(I)$, there are finitely many
points $x_{2i} \in \alpha(I)$,  $1\leq i \leq m$, 
such that 
\be
\alpha(I) \in  \bigcup_{1}^m \interior(J^-(x_{2i}))\cup \interior(J^+(x_{2i})).
\ee
By connectedness, and by reordering
if necessary, we may assume that 
\be\label{eq-p2}
p \in
\interior(J^-(x_{2}))\cup \interior(J^+(x_{2})),
\ee
\be\label{eq-p2m}
q \in 
\interior(J^-(x_{2m}))\cup \interior(J^+(x_{2m})),
\ee
and for $1 \leq i \leq m-1$
\be\label{eq-p2i-p2i+2}
[\interior(J^-(x_{2i}))\cup \interior(J^+(x_{2i}))]
\cap [\interior(J^-(x_{2i+2}))\cup \interior(J^+(x_{2i+2}))]
 \neq \emptyset.
 \ee
Since $X$ is causally connected, 
and \eqref{eq-p2} holds, there is a causal curve $\gamma_1$ from $p$ to
$x_2$. Similarly, for $1< i < m$, since there is a point
$x_{2i+1}$ contained in the set given in \eqref{eq-p2i-p2i+2}, there is a causal curve
$\gamma_{2i}$ from $x_{2i}$ to $x_{2i+1}$, and a causal curve
$\gamma_{2i+1}$ from $x_{2i+1}$ to $x_{2i+2}$. Finally, by \eqref{eq-p2m}, there is a causal curve $\gamma_{2m}$ from $x_{2m}$
to $q$. Then $\gamma = \gamma_1 \cdots \gamma_{2m}$ is a piecewise causal curve
from $p$ to $q$. 
\end{proof}

The following example shows that, in general, we should not expect the property of being causally null to be stable under intrinsic timed-Hausdorff, nor timed Gromov--Hausdorff, convergence.
\begin{example}\label{ex:causally null not preserved}
For $j \in \mathbb N$, let $X_j$ be set
\be
X_j = \left\{
\left(\frac{k}{j},0\right)
\right\}_{k=0}^{j}
\cup\left\{
\left(\frac{2k-1}{2j},0\right)
\right\}_{k=1}^{j}\subset \mathbb{R}^2
\ee
endowed with the metric
\be
d_j((x,y),(x',y')) = |x-x'|
\ee
and the time function
\be
\tau_j(x,y) = y.
\ee
It is easy to verify that $(X_j,d_j,\tau_j)$ is causally null and that $(X_j,d_j,\tau_j)\tHto (X,d,\tau)$ where
\be
X = [0,1], \quad d(x,x') = |x-x'|, \quad \tau(x) = 0. 
\ee
However, the null distance in $X$ is infinite, since for any $x\neq x'$, $x,x'\in X$ there is no finite (not even a countably infinite) chain of causally related points connecting $x$ to $x'$. Therefore, $X$ is not causally null.
\end{example}

\begin{rmrk}
One may argue that the previous counterexample is pathological because the spaces $X_j$ are discrete, or because the time function in the limit is identically zero. However, it is possible to construct counterexamples where the sequence consists of (regular) space-times. This will be done in the upcoming work \cite{prados-zoghlami}.
\end{rmrk}

We conclude this section by studying the relationship between 
null distances that encode causality in the sense of \cite{SakSor-Null} and our definition of causally-null timed-metric spaces.

\begin{prop}\label{prop:causal}
Let $N$ be a space-time and consider a time function $\tau\colon N\to [0,\infty)$ whose null-distance encodes causality in the sense of \cite{SakSor-Null} and defines a timed-metric space $(N, \hat d_{\tau},\tau)$. Then $(N, \hat d_{\tau},\tau)$ is causally null in the sense of Definition \ref{defn:null-distance}.
\end{prop}

\begin{proof}
Let $\hat{d}_{\tau}$ be the null distance on $N$ associated to the time function $\tau$. The fact that $\hat{d}_{\tau}$ encodes causality means that $x\leq_N y$, where $\leq_N$ denotes the causal order in $N$, if and only if 
\[
\tau(y)-\tau(x) = \hat{d}_{\tau}(x,y).
\] However, this 
coincides with the definition of causality
for the timed-metric space $(N,\hat{d}_{\tau},\tau)$, see \eqref{def-causal-struc}. Therefore,
\begin{align*}
\hat{d}_{\hat{d}_{\tau},\tau}(x,y) &= \inf\left\{
\sum_{i=1}^{N} |\tau(p_i)-\tau(p_{i-1})| : x = p_0, \dots, p_{N}=y,\ p_{i}\in J^\pm(p_{i-1})
\right\}\\
&= \inf\left\{
\sum_{i=1}^{N} |\tau(p_i)-\tau(p_{i-1})| : x = p_0, \dots, p_{N}=y,\ p_{i}\leq_N p_{i-1}\ \text{or}\ p_{i-1}\leq_N p_i
\right\}\\
&= \hat{d}_{\tau}(x,y),
\end{align*}
for any $x,y\in N$, and the claim follows.
\end{proof}

\begin{rmrk}
If $N$ is a path-connected space-time and 
$\tau\colon N \to \mathbb R$ is a generalized proper time function which is locally anti-Lipschitz in the sense of \cite{chrusciel-grant-minguzzi}, then, 
by \cite{SakSor-Null}, $(N, \hat d_{\tau}, \tau)$ is a timed-metric space and $\hat{d}_{\tau}$ encodes causality. In particular, $(N, \hat d_{\tau}, \tau)$ is a causally-null timed-metric space in the sense of Definition~\ref{defn:null-distance}. 
\end{rmrk}

\begin{example}
    Observe that, in general, Proposition~\ref{prop:causal} does not imply that the completion of $(N, \hat d_{\tau},\tau)$ is causally null in the sense of Definition~\ref{defn:null-distance}, not even when $N$ is a causally-null compactifiable space-time, in the sense of \cite{SakSor-Notions}, and $\tau$ is the cosmological time function, as was pointed out to us by Omar Zoghlami. 
    
    Indeed, let $N$ be the space-time consisting of the interior of the triangle in $\mathbb{R}^2$ with vertices $(0,0)$, $(1,0)$ and $(0.75,0.25)$, endowed with the restriction of the Minkowski metric (where the second coordinate has negative signature). This space-time is clearly causally-null compactifiable, with $\left(\bar{N},\bar{\hat{d}}\right)$ isometric to the corresponding closed triangle. However, under this isometry, one can see that the point $(0,0)\in \bar{N}$ is not causally related to any other point in $\bar{N}$, in the sense of Definition~\ref{def-causal-struc}, which implies $\hat{d}_{\bar{d},\bar{\tau}}((0,0),(x,y))=\infty\neq \bar{d}((0,0),(x,y))$ for any $(x,y)\in \bar{N}$, where $\bar{\tau}$ and $\bar{d}$ are the canonical extensions of the cosmological time function and the associated null distance in $N$ to $\bar{N}$. 
\end{example}
\begin{rmrk}
    Conditions granting that the metric completion of a causally-null timed-metric space is also causally null will be addressed in the upcoming work \cite{prados-zoghlami}.
\end{rmrk}

\appendix

\section{Gromov's compactness theorem}\label{sec-appendix}

Here we reprove Gromov's Compactness Criterion from \cite{Gromov-poly} adding as many details as possible. Note that as a straightforward application, we gave the proof of our compactness theorem in Section \ref{sec-timedGH}.

\begin{thm}\label{thm-GHcriterion}
If $(X_j,d_j)$ is a sequence of compact timed-metric spaces that is equibounded and equicompact, as in \eqref{eq:equibounded-B} and \eqref{eq:equicompact-B}
Then there exists a subsequence of the $(X_j, d_j)$ that converges in the Gromov--Hausdorff sense to a compact metric space $(Y_\infty, d_\infty)$. In fact, there exist compact $F$ and $Y_\infty$, distance preserving maps, up to a subsequence that we do not relabel,
\[
h_j: X_j \to F \qquad \forall j\in \mathbb N,
\]
and $h_\infty : Y_\infty \to  F$ distance preserving, 
 such that 
\begin{equation}\label{eq-GHcriterion-H}
    d_H^{F} (h_j(X_j),h_\infty(Y_\infty))\to 0.
\end{equation}
\end{thm}

\bigskip 

\begin{proof}
Apply Proposition \ref{prop:selection}
to obtain a
countable index set $A=\disjointunion A_i$,  and 
a countable dense collection of points,
\be
I^j(A)=\{x^j_{a}=I^j(a):\,a\in A\}\subset X_j \qquad j \in \mathbb N,
\ee
that satisfy \eqref{eq:cover} and \eqref{eq:in-0}.

Let 
\be
F'= \Bigl\{ f : A \to \mathbb{R} \, : f \,\text{bounded} \Bigr\} 
\ee
endowed with the sup norm,
$\|f\| = \sup_{a \in A} |f(a)|$. Then define 
\be
F = \Bigl\{ f \in F'  \, | f \text{ satisfies  \eqref{eq-fA_1bdd} and \eqref{eq-fcont}} \Bigr\},
\ee
where conditions  \eqref{eq-fA_1bdd} and \eqref{eq-fcont} are given by 
 \be\label{eq-fA_1bdd}
f(A_1) \subset [0,D]
 \ee
 and 
 \be\label{eq-fcont}
|f(a) - f(p_i(a))| \leq 2\varepsilon_i \,\, \forall i \in \mathbb N\,\forall a \in A_{i+1}.
 \ee

\medskip

We first claim that for each $x \in X_j$, the map $(h_j(x)) : A \to \mathbb R$ given by  
\be\label{eq-hjxa}
(h_j(x))(a) =d_j(x, I^j_i(a)), \qquad \forall a_i \in A_i \subset A
\ee
is contained in $F$. 
By  the equiboundedness condition, \eqref{eq:equibounded-B}, 
it follows that $(h_j(x))\leq D$.
Thus, $(h_j(x)) \in F'$, and in particular, $(h_j(x))$ satisfies \eqref{eq-fA_1bdd}.

Now, by \eqref{eq:in-0}, i.e. for each $a \in A_{i+1}$,  
\[
I^j_{i+1}(a) \in B_{d_j}(I^j_i(p_i(a)), 2\varepsilon_i)\subset X_j,
\]
and the triangle inequality, we get
\[
|(h_j(x))(a) - (h_j(x))(p_i(a))|= |d_j(x, I^j_{i+1}(a)) - d_j(x, I^j_i(p_i(a)))| \leq 2 \varepsilon_i.
\]
Hence, $h_j(x) \in F$. 

\medskip

Secondly, we claim that the maps $h_j : X_j \to F$ given by \eqref{eq-hjxa} are distance preserving. Indeed, since the collection of points,
$I^j(A)=\bigcup_{i\in {\mathbb N}} I^j_i(A_i)$ is dense in $X_j$, 
for any $x \in X_j$ there exists a sequence $a^k \in A$ such that $d_j(x,I^j(a^k)) \to 0$ as $k \to \infty$. 
Thus, for  $y \in X_j$ 
\[
|(h_j(x))(a^k) - (h_j(y))(a^k)|= |d_j(x, I^j(a)) - d_j(y, I^j(a^k))| \to d_j(y, x).
\]
This gives 
$\| h_j(x) - h_j(y)\| \leq d_j(y, x)$. The reversed inequality follows by the triangle inequality. Hence, $h_j$ are distance preserving.

\medskip

Now we claim that $F$ is compact. To show this, take a sequence of functions $f_j \in F$. We now construct a pointwise limit function by induction on $i$.
For $i=1$, 
since $f_j(A_1) \subset [0,D]$, there is a subsequence of the $f_j$, that we do not relabel, 
and 
\be
f_\infty: A_1 \to [0,D]
\ee
such that 
\be
f_j \to f_\infty  \quad \text{uniformly in} \quad A_1,
\ee
where the uniform convergence is satisfied since $A_1$ is finite. 
Note that by \eqref{eq-fcont}, one has 
\begin{equation}\label{eq-fcont2}
    -2\varepsilon_i + f_j(p_i(a)) \leq f_j(a) \leq 2\varepsilon_i + f_j(p_i(a)) \qquad \forall a \in A_{i+1}. 
\end{equation}
In particular, by \eqref{eq-fA_1bdd}, we get 
 \be
-2\varepsilon_1 \leq f_j(a) \leq 2\varepsilon_1 + D \qquad \forall a \in A_2. 
\ee
Thus, we can find a further subsequence of the $f_j$, that we do not relabel, and extend $f_\infty$ to a function 
\be
f_\infty: A_1 \disjointunion A_2 \to [-2\varepsilon_i,2\varepsilon_i+D]
\ee
 such that 
\be
f_j \to f_\infty  \quad \text{uniformly in} \quad  A_1\disjointunion A_2,
\ee
where the uniform convergence is satisfied since $A_1 \disjointunion A_2$ is finite. This 
$f_\infty$ still satisfies \eqref{eq-fA_1bdd}, and for $i=1$ taking the limit in \eqref{eq-fcont2}, it also satisfies \eqref{eq-fcont} for $i=1$. 

Now assume that we have 
\be
f_\infty : A_1 \disjointunion \cdots \disjointunion A_k \to [ -2\sum_{i=1}^{k-1}\varepsilon_i, 2\sum_{i=1}^{k-1}\varepsilon_i +D ]
\ee
satisfying \eqref{eq-fA_1bdd} and, \eqref{eq-fcont} for $i=1,\ldots,k-1$, such that
\be
f_j \to f_\infty  \quad \text{uniformly in} \quad  A_1 \disjointunion \cdots \disjointunion A_k,
\ee
We now extend $f_\infty$ to $A_{k+1}$.
For $a \in A_{k+1}$, recursively applying condition (\ref{eq-fcont2}) for $i=k-1,\ldots,1$, we get 
\be
-2\sum_{i=1}^{k}\varepsilon_i  \leq f_j(a) \leq  2\sum_{i=1}^{k}\varepsilon_i+ D. 
\ee
Thus, we can find a further subsequence of the $f_j$, that we do not relabel, and a function 
\be
f_\infty: A_{k+1} \to [-2\sum_{i=1}^{k}\varepsilon_i , 2\sum_{i=1}^{k}\varepsilon_i+D]
\ee
 such that 
\be
f_j \to f_\infty  \quad \text{uniformly in} \quad  A_1 \disjointunion \cdots \disjointunion A_{k+1},
\ee
where the uniform convergence is satisfied since $
A_1 \disjointunion \cdots \disjointunion A_{k+1}
$ is finite. This 
$f_\infty$ still satisfies \eqref{eq-fA_1bdd}, and for $i=k$ taking the limit in \eqref{eq-fcont2}, it also satisfies \eqref{eq-fcont} for $i=k$. 
This gives the required function $f_\infty \in F$.

\medskip
We now show that the resulting subsequence $f_j$ converges to $f$ with respect to the sup norm.
Given $\varepsilon>0$, choose $m \in \mathbb N$ such that $\sum_{i=m}^\infty \varepsilon_i < \varepsilon/8$. Since by construction $f_j$ converges uniformly to $f_\infty$ in $
A_1 \disjointunion \cdots \disjointunion A_{m}$, there exists $J(\varepsilon) \in \mathbb N$ such that for all $j \geq J(\varepsilon)$
\begin{equation}\label{eq-mbound}
    |f_j(a)- f_\infty(a)| < \varepsilon/2 \qquad \forall a \in A_1 \disjointunion \cdots \disjointunion A_{m}. 
\end{equation}
Recursively applying condition (\ref{eq-fcont2}) for $i=m+k-1,\ldots,m$, we get 
\be
-2\sum_{i=m}^{m+k-1}\varepsilon_i + f_j(p_m(a)) \leq f_j(a) \leq  2\sum_{i=m}^{m+k-1}\varepsilon_i+ f_j(p_m(a)),
\ee
and similarly for $f_\infty$. 
Substracting the corresponding inequalities and using \eqref{eq-mbound}, we get  
\be
|f_j(a)- f_\infty(a)| < \varepsilon/2 + 4 \sum_{i=m}^{m+k-1} \varepsilon_i < \varepsilon \qquad \forall a \in A_{m+k},\,\forall j \geq N(\varepsilon). 
\ee
This together with \eqref{eq-mbound} gives 
\be
|f_j(a)- f_\infty(a)| < \varepsilon \qquad \forall a \in A, \,\forall j \geq N(\varepsilon).. 
\ee
Hence, $f_j$ converges to $f$ with respect to the sup norm.

\medskip 
Finally, since $F$ is compact, 
by Blaschke's compactness theorem (c.f. Chapter 7 in \cite{BBI}),
there is a subsequence of the $X_j's$, that we do not relabel,
such that $h_j(X_j)$ converges in Hausdorff sense in $F$ to some compact metric space $Y_\infty \subset F$.
\end{proof}

\bibliographystyle{alpha}
\bibliography{CP-timed_metric_spaces_and_causality}
\end{document}